\documentclass[english,12pt]{amsart}
\usepackage{amsmath}
\usepackage{amsthm}
\usepackage{amssymb}
\usepackage{amsfonts}
\usepackage{geometry,graphicx} 
 \usepackage{amscd}
 \usepackage{color,xcolor}
 \usepackage{pdfsync}
\usepackage[ pdfborderstyle={/S/U/W 1}]{hyperref} 
\usepackage{url}
\usepackage{tikz,pgfplots}
\usepackage{caption}
\usepackage{subcaption}
\usepackage{tikz}

\usetikzlibrary{calc}

\usepackage{booktabs}
\usepackage{colortbl}
 \definecolor{Ftitle}{RGB}{11,46,108}
\definecolor{line}{RGB}{87,39,117}
\colorlet{tableheadcolor}{Ftitle!25} 
\colorlet{tablerowcolor}{gray!10} 

\newtheorem{theorem}{Theorem}

\newtheorem{corollary}[theorem]{Corollary}
\newtheorem{conjecture}[theorem]{Conjecture}
\newtheorem{proposition}[theorem]{Proposition}
\newtheorem{definition}[theorem]{Definition}

\newtheorem{example}[theorem]{Example}
\newtheorem{remark}[theorem]{Remark}

\definecolor{MyBlue}{HTML}{210cac}
\definecolor{MyCiteColor}{HTML}{0099FF}
\definecolor{bl1}{HTML}{4479A1}
\definecolor{pur1}{HTML}{52196D}
\definecolor{mag1}{HTML}{2AD0F1}
\definecolor{MyRed}{HTML}{3E186A}
\hypersetup{%
  pdfborderstyle={/S/U/W 1},
  colorlinks=true,
  linkcolor=MyBlue,
  citecolor=MyCiteColor,
  urlcolor=MyRed,
  anchorcolor=MyBlue,
  linkbordercolor=MyRed,
}

\newcommand{\bC}{\mathbb C}

\newcommand{\bP}{\mathbb P}
\newcommand{\bQ}{\mathbb Q}

\newcommand{\bZ}{\mathbb Z}

\newcommand{\cH}{\mathcal H}

\newcommand{\V}{V}

\newcommand{\Vsing}{V_{\text{sing}}}
\newcommand{\Vreg}{V_{\text{reg}}}

\newcommand{\mldeg}{\mathrm{mldeg}}

\newcommand{\EAf}{E_A(f)}
\newcommand{\LocusEAf}{\Sigma_A}

\DeclareMathOperator{\conv}{conv}

\usepackage{floatrow}
\newfloatcommand{capbtabbox}{table}[][\FBwidth]

\usepackage{blindtext}

\begin{document}

\title{The Maximum Likelihood Degree of Toric Varieties}

\title[The ML Degree of Toric Varieties]{The Maximum Likelihood Degree of Toric Varieties}

\author[Likelihood Geometry Group]{Carlos Am\'endola, Nathan Bliss, Isaac Burke, Courtney R. Gibbons, Martin Helmer, Serkan Ho\c{s}ten, Evan D. Nash, Jose Israel Rodriguez, Daniel Smolkin}
\thanks{This project was started during the AMS Mathematics Research Communities 2016 program in Algebraic Statistics held in Snowbird, Utah. The authors were part of the Likelihood Geometry group led by Serkan Ho\c{s}ten and Jose Rodriguez. This material is based upon work supported by the National Science Foundation under Grant Number DMS 1321794.}









\begin{abstract}
We study the maximum likelihood (ML) degree of toric varieties, known as discrete exponential models in statistics. 
By introducing scaling coefficients to the monomial parameterization of the toric variety, one can change the ML degree. 
We show that the ML degree is equal to the degree of the toric variety for generic scalings, while it drops if and only if the scaling vector is in the locus of the principal $A$-determinant.
We also illustrate how to compute the ML estimate of a toric variety numerically via homotopy continuation from a scaled toric variety with low ML degree.
Throughout, we include examples motivated by algebraic geometry and statistics. We compute the ML degree of rational normal scrolls  and a large class of Veronese-type varieties. In addition, we investigate the ML degree
of scaled Segre varieties, hierarchical log-linear models, and graphical models.
\end{abstract}


\maketitle

\section{Introduction}

Maximum likelihood estimation is a fundamental optimization
problem in statistical inference. The maximum likelihood degree of an algebraic statistical model was introduced in \cite{CHKS06} and \cite {HKS05} to 
study the geometry and complexity of this optimization problem. Here we study the maximum likelihood degree of toric varieties. 

Let $\V \subset \bP^{n-1}$ be a projective variety over $\bC$. The homogeneous coordinates of $\bP^{n-1}$ will be denoted by 
$[p_1 \, : \, \cdots \, : \, p_n]$. The affine open subset of $\V$ where $p_1 + \ldots + p_n \neq 0$ is identified with the set of points in $\V$ where
$p_1 + \ldots + p_n = 1$. We can view the points $p \in \V$ satisfying
$$ p_1, \ldots, p_n \geq 0 \mbox{   and  } p_1 + \ldots + p_n =1$$
as a family of probability distributions of a discrete random variable $X$ where 
$$p_i = \mathrm{prob}(X=i) \,\,\,\,   \mbox{for  } i=1, \ldots, n.$$ After observing $X$ in $N$ instances $X_1, \ldots, X_N$, we record the data vector $u = (u_1, \ldots, u_n) \in~\bZ^n$ where 
$u_i$ is the number instances where $X_j=i$, i.e.,  
$$u_i= |\{j \, : \,   X_j=i \}|\text{ and }N=u_1+\cdots+u_n.$$
The likelihood function is the rational function 
$$ \ell_u(p) \, = \, \frac{p_1^{u_1} \cdots p_n^{u_n}}{(p_1 + \cdots + p_n)^{u_1 + \cdots + u_n}},$$
and one seeks to find a probability distribution $\hat{p} = (\hat{p}_1, \ldots, \hat{p}_n)$ in $V$ which maximizes $\ell_u$. Such a probability distribution $\hat{p}$
is a maximum likelihood estimate, and $\hat{p}$ can be identified by computing all critical points of $\ell_u$ on $V$. 

We let $\Vsing$ be the singular locus of $V$ and $\Vreg = V \setminus \Vsing$ be the regular locus. We also define
$\cH$ as the hypersurface that is the union of the coordinate hyperplanes and the hyperplane defined by $p_1 + \cdots + p_n = 0$.
\begin{definition} \label{MLdeg} The \textit{maximum likelihood degree} of $V$, denoted $\mldeg(V)$,  is the number of complex critical points of the likelihood function $\ell_u$ on 
$\Vreg \setminus \cH$ for generic vectors $u$.
\end{definition}
In \cite{CHKS06} and \cite{HKS05}, it was shown that $\mldeg(V)$ is well-defined. We refer the reader to \cite{DSS09, Huh13, HRS14} for a glimpse of the growing body of work on aspects of $\mldeg(V)$. For an excellent recent survey we recommend \cite{HS14}.

In most applications, $V$ is the Zariski closure of the image of a polynomial map, also known as a parametric statistical model. A widely used subclass
of such models is composed of those given by a monomial parametrization. In algebraic geometry, they are known as toric varieties \cite{Ful93, CLS11}, and in probability
theory and statistics they are known as discrete exponential families \cite{BD76}. In particular, hierarchical log-linear models \cite{BFH07} and undirected graphical
models on discrete random variables \cite{Lau96} present examples in contemporary applications. Theorem~\ref{mainthm} is our main result and is restated below.

\begin{theorem}[Main result] \label{thm-main} Let $V^c \subset \bP^{n-1}$ be the projective variety defined by the monomial parametrization
$\psi^c \, : \, (\bC^*)^d \longrightarrow (\bC^*)^n$ where 
$$ \psi^c(s,\theta_{1},\theta_{2},\dots,\theta_{d-1})  \, = \, (c_{1}s\theta^{a_{1}},c_{2}s\theta^{a_{2}},\dots,c_{n}s\theta^{a_{n}}),$$
and $c \in (\bC^*)^n$ is fixed. 
Then $\mldeg(\V^c) < \deg(\V)$ if and only if $c$ is in the principal $A$-determinant of the toric variety $\V = \V^{(1,\ldots,1)}$.
\end{theorem}

To set the stage, we start with an example chosen from the theory of graphical models. This is the smallest example of a graphical model that is not decomposable. Models that are decomposable have a unique maximum likelihood estimate that is a rational function of the data vector $u$. This is equivalent to $\mldeg(V) = 1$ \cite{GMS06}. 

\begin{example}[Binary $4$-Cycle] \label{ex:binary-4-cycle}
This example is from a study to determine risk factors for coronary heart disease based on data collected in Czechoslovakia \cite{R81, EH85}. Six different
factors affecting coronary heart disease were recorded for a sample of  $1841$ workers employed in the Czech automotive industry. We will only use the following four factors:
whether they smoked ($S$), whether their systolic blood pressure was less than $140$ mm ($B$), whether there was a family history of coronary heart disease ($H$), and 
whether the ratio of beta to alpha lipoproteins was less than $3$ ($L$). Each is a binary variable, and the joint 
random variable $X = (S, B, H, L)$ has a state space of cardinality $16$. The data set is summarized in Table~\ref{table-auto}.

\begin{figure}
\begin{center}
\centering\hspace{-25mm}
\begin{floatrow}
\ffigbox{%
  \begin{tikzpicture}[scale=.5]
\def \radius {3cm}
\def \n {4}

\node[circle,inner sep=1pt,draw](S) at ({45 + (360/\n)}:\radius) {$S$};

\node[circle,inner sep=1pt,draw](B) at ({45 + (360/\n*2)}:\radius) {$B$};

\node[circle,inner sep=1pt,draw](H) at ({45 + (360/\n*3)}:\radius) {$H$};

\node[circle,inner sep=1pt,draw](L) at ({45 + (360/\n*4)}:\radius) {$L$};

\draw (S) -- (B) -- (H) -- (L) -- (S);
\end{tikzpicture}%
}{%
  \caption{\label{fig:binary-4-cycle} The $4$-cycle.}%
}

\capbtabbox{%
  \begin{tabular}{|ccc|cc|}
\hline
$H$ & $L$ &  $B$ &  $S$: no & $S$: yes \\
\hline
neg &  $< 3$ &  $< 140$ &  297  & 275 \\
       &            &   $\geq 140$ & 231  & 121 \\
       &  $\geq 3$ & $<140$ & 150 & 191 \\
      &            &   $\geq 140$ & 155  & 161 \\
pos &  $< 3$ &  $< 140$ &  36  & 37 \\
       &            &   $\geq 140$ & 34  & 30 \\
       &  $\geq 3$ & $<140$ & 32 & 36 \\
      &            &   $\geq 140$ & 26  & 29 \\
\hline
\end{tabular}
}{%
  \caption{\label{table-auto} Worker data.}%
}
\end{floatrow}
\end{center}
\end{figure}

We will use the graphical model known as the $4$-cycle (Figure~\ref{fig:binary-4-cycle}). If we set 
$$p_{ijk\ell}=\mathrm{prob}(S = i, B=j, H = k, L =\ell) $$
the family of probability distributions for $X$ which factor according
to this model can be described by the following monomial parametrization: let $a_{ij}, b_{jk}, c_{k\ell}, d_{i \ell}$
be parameters for $i,j,k, \ell \in \{0, 1 \}$ and let $p_{ijk\ell} = a_{ij}b_{jk}c_{k\ell} d_{i \ell}$. The Zariski closure $\V$ of the image of this parametrization is a toric variety. Every probability distribution 
for $X$ that lies on this toric variety satisfies certain independence statements which can be read from the graph. For instance, $S$ and $H$
are independent given $B$ and $L$. Similarly, $B$ and $L$ are independent
given $S$ and $H$. 
The maximum likelihood estimate given by the worker data is
\begin{align*}{\hat p} = (\, &0.15293342, \, 0.089760679, \, 0.021266977, \, 0.015778191,\\
&0.12976986, \, 0.076165372, \, 0.020853199, \, 0.015471205, \\
&0.13533793, \, 0.11789409, \, 0.018820142, \,  0.0207235, \\
&0.083859917, \, 0.073051125, \, 0.01347576, \, 0.014838619 \, ).
\end{align*}The degree of $\V$ is $64$ and $\mldeg(\V) = 13$. This was first computed in \cite[p.~1484]{GMS06} where the question of explaining the fact that $\mldeg(\V) = 13$
was raised.  
\end{example}

The  article is organized as follows. 
In Section~\ref{sec:prelims} we define scaled toric varieties, introduce maximum likelihood degrees, and recall discriminants. 
In Section~\ref{sec:principal} we prove our main theorem (Theorem~\ref{mainthm}). 
In Sections~\ref{sec:scroll}-\ref{sec:graphical} 
we compute the ML degree of rational normal scrolls  and a large class of Veronese-type varieties. We investigate the ML degree
of scaled Segre varieties, hierarchical log-linear models, and graphical models.
In Section~\ref{sec:homotopy}, we show that maximum likelihood estimates can be tracked between different scalings of a toric variety via homotopy continuation. We illustrate with computational experiments that homotopy continuation can be a better numerical method
than iterative proportional scaling \cite{DR72}
for computing the maximum likelihood estimate of a discrete exponential
model.

\section{Preliminaries}\label{sec:prelims}
In this section we introduce scaled toric varieties and their likelihood equations. We show that the ML degree of a scaled toric variety is bounded by the degree of the variety. Next, we present two classical objects in toric geometry: the $A$-discriminant and the principal $A$-determinant \cite{GKZ94}. 
The principal $A$-determinant plays a crucial role in our main result Theorem~\ref{mainthm}.

\subsection{Scaled toric varieties}

\vskip 0.1cm
\noindent In this paper, we study the maximum likelihood degree of projective toric varieties. Let $A$ be a $(d-1) \times n$ matrix with columns $a_1, \ldots, a_n \in \bZ^{d-1}$. The convex hull of the lattice points $a_i$ defines a polytope $Q=\conv(A)$. 
\begin{definition}\label{def:scaledToric}
A toric variety  scaled by $c \in (\bC^*)^n$, denoted $V^c$, is defined by the following map 
$\psi^c \, : \, (\bC^*)^d \longrightarrow (\bC^*)^n$ where $A$ is a full rank $(d-1)\times n$ integer matrix: 
\begin{equation}
\psi^c(s,\theta_{1},\theta_{2},\dots,\theta_{d-1})  \, = \, (c_{1}s\theta^{a_{1}},c_{2}s\theta^{a_{2}},\dots,c_{n}s\theta^{a_{n}}).\label{eq:scaledEmbedding}
\end{equation}Here $V^c$ is the projective variety in $\bP^{n-1}$ corresponding to the affine cone that is the closure of $\psi^c((\bC^*)^d)$ in $\bC^n$.
\end{definition}
Different choices of $c$ give different embeddings of isomorphic varieties $V^c$ in $\mathbb{P}^{n-1}$; their ML degrees may also differ. When $c=(1,1,\dots,1,1)$, $V^c$ is the standard
toric variety of $A$ as in \cite{CLS11}. 
\begin{example} Consider $\psi^c \, : \, (\bC^*)^3 \longrightarrow (\bC^*)^6$ given by 
$$ (s, t, u) \mapsto (s, 2st, st^2, 2su, 2stu, su^2).$$
This gives the Veronese embedding of $\bP^2$ into $\bP^5$ scaled by $c = (1,2,1,2,2,1)$. The usual Veronese embedding
has ML degree four whereas this scaled version has ML degree one.
\end{example}

\subsection{Likelihood equations}
We begin our discussion with the critical equations whose solutions are the critical points of the likelihood function for a toric variety. 
These equations are called likelihood equations. 

Let $A$ be a $(d-1)\times n$ matrix as in Definition~\ref{def:scaledToric}.
We let 
$$f  =
\sum_{i=1}^n c_{i}\theta^{a_{i}}=c_{1}\theta^{a_{1}}+\cdots+ c_{n}\theta^{a_{n}}.$$

\begin{definition} \label{MLequations}
Let $u =\left(u_{1},\dots,u_{n}\right)$ and $u_+ = \sum_i u_i$. Using the method of Lagrange multipliers, we obtain the likelihood equations
for the variety $V^c$:
\[
1=sf
\]
\[
\left(Au\right)_{i}=u_{+}s\theta_{i}\frac{\partial f}{\partial\theta_{i}}\text{ for }i = 1, \ldots, d-1.
\]
\noindent
In other words,
\begin{equation}
\begin{array}{ccc} \label{likeqn1}
1 & = & sf\\[.5em]
\left(Au\right)_{1} & = & u_{+}s\theta_{1}\frac{\partial f}{\partial\theta_{1}}\\[.5em]
\left(Au\right)_{2} & = & u_{+}s\theta_{2}\frac{\partial f}{\partial\theta_{2}}\\
 & \vdots\\
\left(Au\right)_{d-1} & = & u_{+}s\theta_{d-1}\frac{\partial f}{\partial\theta_{d-1}}
\end{array}
\end{equation}
\end{definition}

\vskip 0.3cm
\noindent We remark that the solutions to the above system are exactly the solutions to 
\begin{equation}
\begin{array}{ccc} \label{likeqn2}
(Au)_1 f  & = & u_{+} \theta_{1}\frac{\partial f}{\partial\theta_{1}}\\[.5em]
(Au)_2 f & = & u_{+} \theta_{2}\frac{\partial f}{\partial\theta_{2}}\\
 & \vdots\\
(Au)_{d-1} f & = & u_{+}\theta_{d-1}\frac{\partial f}{\partial\theta_{d-1}}
\end{array}
\end{equation}
where $f \neq 0$.

\vskip 0.3cm
\noindent We formulate the ML degree of the scaled toric variety $V^c$ using the coordinates $p_1, \ldots, p_n$ as well. Note that 
$f = c_1\theta^{a_1} + \cdots + c_n \theta^{a_n}$ induces a linear form $L_f(p) = c_1p_1 + \cdots + c_np_n$
on the toric variety $V$. Similarly, $\theta_i\frac{\partial f}{\partial \theta_i}$ induce corresponding linear forms $L_i(p)$ 
on $V$ for $i=1, \ldots, d-1$, as the monomials appearing in $\theta_i \frac{\partial f}{\partial \theta_i}$ are a subset of the monomials appearing in $f$. With this we immediately get the following.
\begin{proposition} The ML degree of $V^c$ is the number of solutions $p \in V \setminus V(p_1\cdots p_n L_f(p))$
to 
\begin{equation} \label{eqlik3}
\begin{array}{ccc}
(Au)_1 L_f(p)  & = & u_{+} L_1(p)\\
(Au)_2 L_f(p) & = & u_{+} L_2(p)\\
 & \vdots\\
(Au)_{d-1} L_f(p) & = & u_{+} L_{d-1}(p)
\end{array}
\end{equation}
for generic vectors $u$.
\label{prop:eqlik3}
\end{proposition}
\begin{proof}
Definition~\ref{MLdeg} combined with \eqref{likeqn2} give the result by noting the fact that the singularities of any toric variety are contained in the union of coordinate hyperplanes.
\end{proof}
The following result was first proved in \cite[Theorem 3.2]{HS14}.
\begin{corollary} For $c \in (\bC^*)^n$, the ML degree of $V^c$ is at most $\deg(V)$. Moreover, for generic $c$, $\mldeg(V^c) = \deg(V)$.
\label{cor:mldegIneq}
\end{corollary}
\begin{proof} For generic $u$, the linear equations \eqref{eqlik3} define a linear subspace of codimension
$d-1$. By Bertini's theorem, the intersection of this linear subspace with $V$ is generically transverse. Now by Bezout's theorem \cite[Propositon 8.4]{Ful98}, 
the sum of the degrees of the components of the intersection  is equal to the degree of $V$. Therefore, the degree of the zero dimensional piece of this intersection 
is at most $\deg(V)$. The ML degree of $V^c$ is obtained
by removing those solutions in $V(p_1\cdots p_n L_f(p))$, hence $\mldeg(V^c) \leq \deg(V)$. 
For generic $c$,
the linear subspace is a generic linear subspace of codimension $d-1$.
 Not only will the intersection contain
exactly $\deg(V)$ points, these points will not be in $V(p_1 \cdots p_n L_f(p))$. Hence, in this case, $\mldeg(V^c) = \deg(V)$.
\end{proof}

We will denote the set of solutions in $V$ to \eqref{eqlik3}  by $\mathcal{L}_{c,u}'$. Similarly, we will denote the set of solutions to
\eqref{eqlik3} in $\V\setminus V(p_1\cdots p_n L_f(p))$ by $\mathcal{L}_{c,u}$.
To close this subsection, we would like to point out a classical result
both in probability theory and toric geometry. It is sometimes known as Birch's theorem (see e.g. \cite{Lau96}).
\begin{theorem}[Birch's Theorem] \label{thm-Birch}
There is a unique positive point in $\mathcal{L}_{c,u}(p)$, namely, there is 
a unique positive solution to \eqref{eqlik3} in $\V \setminus V(p_1 \cdots p_n L_f(p))$ if the scaling vector $c$ and the data vector $u$ 
are positive vectors.
\end{theorem}
\subsection{$A$-discriminant and principal $A$-determinant}
In this subsection we recall the classical definitions of discriminants and resultants. 

\begin{definition}
To any matrix $A$ as above, we can associate the variety $\nabla_A$, defined by
\[
\nabla_{A} =\overline{\left\{ c\in\left(\mathbb{C}^{*}\right)^{n} \mid\exists\theta\in\left(\mathbb{C}^{*}\right)^{d-1}\text{ such that }f\left(\theta\right)=\frac{\partial f}{\partial\theta_{i}}\left(\theta\right)=0\text{ for all }i\right\} }.
\]
If $\nabla_{A}$ has codimension one in $(\mathbb{C}^*)^n$, then the \emph{$A$-discriminant}, denoted $\Delta_{A}(f)$,
is defined to be the irreducible polynomial that vanishes on $\nabla_{A}$. This polynomial is unique up to multiplication by a scalar.
\end{definition}
As in the previous subsection, to $f = \sum_{i=1}^n c_i \theta^{a_i}$ where $c = (c_1, \ldots, c_n) \in (\bC^*)^n$ we associate 
the linear form $L_f(p) = \sum_{i=1}^n c_i p_i$. 
\begin{definition}
Given $f_1, \ldots, f_d$ where $f_j = \sum_{i=1}^n c_{i, j} \theta^{a_i}$ with  coefficients $c_{i,j} \in \bC^*$ there is 
a unique irreducible polynomial $R_A(f_1,\ldots, f_d)$  in the coefficients $c_{i,j}$, so that 
$R_A(f_1,\ldots, f_d) = 0$ at $(c_{i,j})$ if and only if $L_{f_1} = L_{f_2} = \cdots = L_{f_d} = 0$ on the toric variety $V$ \cite[Section 2, Chapter 8]{GKZ94}. This polynomial
is called the \textit{$A$-resultant} of $f_1, \ldots, f_d$. 
\end{definition} 
\begin{definition} For $f = \sum_{i=1}^n c_{i} \theta^{a_i}$,  we define the \textit{principal $A$-determinant} as
$$\EAf  \, = \, R_A\left(f, \theta_1 \frac{\partial f}{\partial \theta_1}, \ldots, \theta_{d-1} \frac{\partial f}{\partial \theta_{d-1}}\right).$$ 
This is a polynomial in the coefficients $c_{i}$ of $f$.
\end{definition}

In \cite{GKZ94}, the principal $A$-determinant $\EAf$ is related to the polytope $Q$ of the toric variety $\V$.
More precisely, 
when the toric variety $\V$ is smooth, and $A$ is the matrix whose columns correspond to the lattice points in $Q$, the principal $A$-determinant is 
\begin{equation}\label{eq:PrincipalADetProduct}
\EAf \, = \, \prod_{\Gamma \text{ face of } Q} \Delta_{\Gamma \cap A}
\end{equation}
where the product is taken over all nonempty faces $\Gamma \subset Q$ including $Q$ itself \cite[Theorem 1.2, Chapter 10]{GKZ94}. Here, $\Gamma \cap A$ is the matrix whose columns correspond to the lattice points contained in $\Gamma$. 
When $\V$ is
not smooth, the radical of its principal $A$-determinant is the polynomial given in \eqref{eq:PrincipalADetProduct}. 
The zero locus of $\EAf$ in $\bC^n$ will be denoted by $\LocusEAf$.

\section{ML degree and principal $A$-determinant}\label{sec:principal}
We first state and  prove our main theorem.

\begin{theorem} \label{mainthm}
Let $A = (a_1 \, \ldots \, a_n)$ be a full rank $(d-1)\times n$ integer matrix, and let $V^c \subset \bP^{n-1}$ be the scaled toric variety defined by the monomial parametrization given by \eqref{eq:scaledEmbedding} where $c \in (\bC^*)^n$ is fixed. 
Then $\mldeg(\V^c) < \deg(\V^{(1,1,\ldots,1,1)})$ if and only if $c \in \LocusEAf$.
\end{theorem}

\begin{proof} Let $\V = \V^{(1,1,\ldots,1,1)}$.
Let $c \in (\bC^*)^n \cap \LocusEAf$. Then there exists $p \in V$ such that $L_f(p) = L_1(p) = \cdots = L_{d-1}(p) = 0$. 
Such a $p$ is a solution in $\mathcal{L}_{c,u}'$ but not a solution in $\mathcal{L}_{c,u}$. Since the degree of $\mathcal{L}_{c,u}'$ is equal to $\deg(V)$, and since $\mldeg(V^c)$ is equal to the degree of $\mathcal{L}_{c,u}$
we see that $\mldeg(V^c) < \deg(V)$.  

Conversely, suppose $\mldeg(V^c) < \deg(V)$. There are two ways this can happen (c.f. the proof of Corollary \ref{cor:mldegIneq}). On the one hand, there can exist $p \in V$ 
so that $L_f(p) = 0$,  
which implies that $L_1(p) = \cdots = L_{d-1}(p) = 0$. This means $c \in \LocusEAf$. 
On the other hand, there can exist a solution $p$ in $\mathcal{L}_{c,u}'$ where some $p_i = 0$. The coordinates of $p$ that are zero cannot be 
arbitrary. The support of $p$, that is $\{i \,: \, p_i \neq 0\}$, is in bijection with the columns $a_i$ of $A$ where  the convex hull of 
these columns is a face of $Q = \mathrm{conv}(A)$. Let $\Gamma$ be the corresponding face. Without loss of generality we will 
assume that $\Gamma \cap A = \{a_1, \ldots, a_k\}$. Let $e < d-1$ be the dimension of $\Gamma$. 
Since $p_{k+1} = \cdots = p_{n} = 0$, the point $p_\Gamma = [p_1 : \cdots : p_k]$ is in the toric variety $V_\Gamma$ defined by
$\Gamma \cap A$. Moreover, $L_f(p) = L_{f_\Gamma}(p_1, \ldots, p_k)$ where $L_{f_\Gamma} = \sum_{i=1}^k c_{i} p_i$
is the linear form associated to the polynomial $f_\Gamma$ whose support is $\Gamma$. Similarly,
$L_i(p) = L_i'(p_1, \ldots, p_k)$ where $L_i'$ is the linear form associated to $\theta_i \frac{\partial f_\Gamma}{\partial \theta_i}$.
If $L_{f}(p) = L_{f_\Gamma}(p_1, \ldots, p_k)  \neq 0$, then $p_\Gamma$ is a solution 
to  
\[
\begin{array}{ccc}
1 & = & L_{f_\Gamma} \\
(Au)_{1}  & = & u_{+} L_1'\\
(Au)_{2}  & = & u_{+} L_2'\\
 & \vdots\\
(Au)_{d-1} & = & u_{+} L_{d-1}'.
\end{array}
\]
But this is a $d \times k$ linear system of rank $e < d$. Because $u$ is generic, this system does not have 
a solution. Therefore the only way $p = [p_1 : \cdots : p_k : 0 : \cdots : 0]$ can be a solution in $\mathcal{L}_{c,u}'$ is with 
$L_{f_\Gamma}(p_1, \ldots, p_k) = L_f(p) = 0$. We again conclude that $c \in \LocusEAf$.
\end{proof}

\begin{remark} A natural question one can ask is under what conditions a general projective variety has ML degree one. This question was answered in \cite{Huh14}.
\end{remark}

\subsection{Toric Hypersurfaces} \label{sec:hypersurface}

Let $A = \left( a_1 \, \ldots \, a_{d+1}\right)$ be a $(d-1) \times (d+1)$ integer matrix of rank $d-1$. In this case, the toric variety $\V$ and each scaled toric variety $\V^c$ is a 
hypersurface generated by a single polynomial.  This polynomial can be computed as follows. We let $A'$ be the matrix obtained by appending a row of $1$'s to the matrix $A$.
The kernel of $A'$ is generated by an integer vector $w = (w_1,\ldots,w_{d+1})$ where $\gcd(w_1, \ldots,w_{d+1}) = 1$. Without loss of generality we assume that 
$w_1,\ldots,w_\ell >0$ and 
$w_{\ell +1},\ldots, w_{d+1} \leq 0$. Then the hypersurface is defined  by 
 \[
p_1^{w_1} p_2^{w_2} \cdots p_{\ell}^{w_\ell} - p_{\ell+1}^{-w_{\ell+1}} \cdots p_{d+1}^{-w_{d+1}}.
\] 
We say $A'$ and $A$ are in general position if no $d$ columns of $A'$ lie on a hyperplane. The matrix $A'$ is in general position if and only if $w$ has full support, i.e.~no $w_i$ is $0$. 
This is also equivalent to $Q = \conv(A)$ being a simplicial polytope with exactly $d$ of the columns of $A$ on each facet of $Q$. Recall that $f = \sum_{i = 1}^{d+1} c_i \theta_i^{a_i}$.
In this setting, the $A$-discriminant $\Delta_A$ and the principal $A$-determinant $\EAf$ can be 
calculated directly.
\begin{proposition}\cite[Chapter 9, Proposition 1.8]{GKZ94}  If $A$ is in general position then 
\[ \Delta_A = 
\left(w_{\ell+1}^{-w_{\ell+1}} \cdots w_{d+1}^{-w_{d+1}}\right) c_1^{w_1} \cdots c_{\ell}^{w_\ell} - \left(w_1^{w_1} \cdots w_\ell^{w_\ell}\right) c_{\ell+1}^{-w_{\ell+1}} \cdots c_{d+1}^{-w_{d+1}}.
\]
This is also the principal $A$-determinant of this toric hypersurface up to 
a monomial factor. 
\end{proposition}
\begin{corollary}\label{cor-hypersurface} If $A$ is in general position, $\mldeg(\V^c) < \deg(\V)$ 
if and only if 
\[c \in V\left( \left(w_{\ell+1}^{-w_{\ell+1}} \cdots w_{d+1}^{-w_{d+1}}\right) c_1^{w_1} \cdots c_{\ell}^{w_\ell} - \left(w_1^{w_1} \cdots w_\ell^{w_\ell}\right) c_{\ell+1}^{-w_{\ell+1}} \cdots c_{d+1}^{-w_{d+1}}\right)
\]
where $p_1^{w_1} p_2^{w_2} \cdots p_{\ell}^{w_\ell} - p_{\ell+1}^{-w_{\ell+1}} \cdots p_{d+1}^{-w_{d+1}}$ is the polynomial defining the hypersurface $V$.
\end{corollary}

\
\section{Rational Normal Scrolls}\label{sec:scroll}
In this section, we characterize the ML degree of a rational normal scroll.
A {\it rational normal scroll} is a toric variety associated to a $(d-1) \times n$ matrix $A$ of the following form:
\begin{equation}\label{ScrollMatrix}
A = \left( \begin{array}{ccccccccccccc}
1 & \cdots & 1 & 0 & \cdots & 0 & \cdots & 0 & \cdots & 0 & 0 & \cdots & 0 \\
0 & \cdots & 0 & 1 & \cdots & 1 & \cdots & 0 & \cdots & 0 & 0 & \cdots & 0 \\
\vdots &&&&&&&&&&&& \vdots \\
0 & \cdots & 0 & 0 & \cdots & 0 & \cdots & 1 & \cdots & 1 & 0 & \cdots & 0 \\
0 & 1\cdots & n_1 & 0 & 1\cdots & n_2 & \cdots & 0 & 1\cdots & n_{d-2} & 0 & 1\cdots & n_{d-1}
\end{array}
\right).
\end{equation}
Here, $n = d-1 + n_1 + n_2 + \cdots + n_{d-1}$.
This description is due to Petrovi\'c \cite{Pet08}. 

We index the data vector and scaling vector as $(u_{ij})$ and $(c_{ij})$ respectively.
For $i=1,2,\dots,d-1$, define $g_i$ to be the following univariate polynomial in the variable $\theta_{d-1}$:
$$g_i = c_{i0}+c_{i1}\theta_{d-1}+c_{i2}\theta_{d-1}^2+\cdots+c_{in_i}\theta_{d-1}^{n_i}.$$ 
Note that for $i\neq d-1$,   $g_i$ is the $i$th partial derivative of $f$ where
$$f = 
(\theta_1g_1+\theta_2g_2+\cdots+\theta_{d-2}g_{d-2})+
g_{d-1}.$$

We now state and prove our result on the ML degree of the rational normal scroll.
\begin{theorem}\label{thm:MLDegRatScroll}
The ML degree of a rational normal scroll given by the matrix $A$ as in  \eqref{ScrollMatrix} is equal to the number of distinct roots of $g_1g_2 \cdots g_{d-1}$.
\end{theorem}
\begin{proof}
For each $i$, we use $g_i'$ to denote the derivative of $g_i$ with respect to $\theta_{d-1}$, and we write $u_{(d-1)+} = \sum_{j=0}^{n_{d-1}} u_{(d-1)j}$. 
With this set up, the last equation in \eqref{likeqn2} is 
\[
(Au)_{d-1}f = u_+ \theta_{d-1} \left( g_{d-1}' + \sum_{i=1}^{d-2} \theta_i g_i'\right).
\]
Multiply both sides of this equation by $g_1 g_2 \cdots g_{d-1}$.
Then substituting $(Au)_if = u_+ \theta_i g_i$ for each $1 \leq i \leq d-2$ in the appropriate summand on the right-hand side gives:
\begin{equation}\label{RatScrollEq}
(Au)_{d-1}g_1\cdots g_{d-1} = \theta_{d-1}\left(u_{(d-1)+} g_1\cdots g_{d-2}g_{d-1}' + \sum_{i=1}^{d-2} (Au)_i g_1 \cdots g_{i-1} g_i' g_{i+1} \cdots g_{d-1} \right).
\end{equation}
For generic $(u_{ij})$, the solutions to \eqref{likeqn2} correspond to the solutions of  \eqref{RatScrollEq} where $g_i \neq 0$ for every $i$.

Suppose that $\theta_{d-1} = \alpha$ is a repeated root of $g_1g_2\cdots g_{d-1}$.
Then either $\alpha$ is a shared root of $g_i$ and $g_j$ for $i \neq j$ or $\alpha$ is a repeated root of a single $g_i$.
In the first case, $\theta_{d-1} - \alpha$ divides the right-hand side of  \eqref{RatScrollEq} since each summand has either $g_i$ or $g_j$ as a factor.
In the second case, $\theta_{d-1}-\alpha$ divides both $g_i$ and $g_i'$ and hence the right-hand side of \eqref{RatScrollEq}.
Factoring out all such repeated roots results in an equation whose degree equals the number of distinct roots of $g_1\cdots g_{d-1}$.
This shows that the ML degree is at most the number of the distinct roots; we must still argue that no more roots can be factored out.

Suppose $g_1g_2\cdots g_{d-1}$ has no repeated roots and let $\theta_{d-1} = \beta$ be a root of an arbitrary $g_i$.
Then $\theta_{d-1} - \beta$ does not divide $g_1\cdots g_{i-1}g_i'g_{i+1} \cdots g_{d-1}$ but does divide every other summand on the right-hand side, so it does not divide the sum as a whole.
This follows because for generic data vectors $(u_{ij})$, no coefficient on the right-hand side vanishes.
Thus, if the repeated roots of $g_1g_2 \cdots g_{d-1}$ have been removed, the degree of the polynomial and hence the ML degree will not drop any further.
\end{proof}

We finish this section with examples.

\begin{example}
The simplest example of a rational normal scroll is a rational normal curve with
$A = \left( 0 \,  1 \, \ldots \, n \right)$.
We denote the scaled rational normal curve of degree $n$ by $\mathcal{C}_n^c$.
The ML degree of the scaled rational normal curve 
$\mathcal{C}_n^c$ is equal to the number of distinct roots of the polynomial
\[ c_0 + c_1 \theta + c_2 \theta^2 + \cdots + c_n\theta^n. \]
For example, if $n = 4$, the zero locus of the principal $A$-determinant (where the ML degree drops from four) is cut out by the discriminant of a degree four polynomial. The variety where the ML degree drops at least two consists of two irreducible components.  The first one corresponds to polynomials of the form $f(\theta) = k(\theta + a)^2 (\theta + b)^2$ and the second corresponds to polynomials of the form $f(\theta) = k(\theta+a)^3(\theta+b)$.
The variety where the ML degree is one is  a determinantal variety, defined by 2-minors of the matrix
$$
\left(
  \begin{array}{ccccc}
    c_1 & 4c_2 & 3c_3 & 8c_4\\
    2c_0 & 3c_1 & c_2 & c_3
  \end{array}
\right).
$$
\end{example}

\begin{example}
In the case that $k=2$, the toric variety is known as a \emph{Hirzebruch surface}, and we can draw the corresponding polytope (Figure \ref{fig:hirzebruch}).
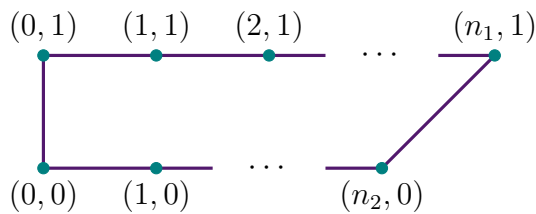
\begin{figure}[h]
  \centering
  \begin{tikzpicture}[scale = 1.5]

    \draw[pur1,very thick] (1.5,0)--(0,0)--(0,1)--(2.5,1);
    \draw[pur1,very thick] (2.5,0)--(3,0)--(4,1)--(3.5,1);
   \draw[fill,teal] (0,0) circle(.05);
    \draw[fill,teal] (1,0) circle(.05);
    \draw[fill,teal] (3,0) circle(.05);
    \draw[fill,teal] (0,1) circle(.05);
    \draw[fill,teal] (1,1) circle(.05);
    \draw[fill,teal] (2,1) circle(.05);
    \draw[fill,teal] (4,1) circle(.05);
    \draw (2,0) node[]{$\cdots$};
    \draw (3,1) node[]{$\cdots$};

    \draw (3,-.25) node[]{$(n_2,0)$};
    \draw (1,-.25) node[]{$(1,0)$};
    \draw (0,-.25) node[]{$(0,0)$};
    \draw (0,1.25) node[]{$(0,1)$};
    \draw (1,1.25) node[]{$(1,1)$};
    \draw (2,1.25) node[]{$(2,1)$};
    \draw (4,1.25) node[]{$(n_1,1)$};
  \end{tikzpicture}
  \caption{Polytope in $\mathbb{R}^2$ corresponding to a Hirzebruch surface.}
  \label{fig:hirzebruch}
\end{figure}
We denote this surface by $\mathcal{H}_{n_1,n_2}$.
For $c = (1,1,\ldots,1,1)$, 
the ML degree of $\mathcal{H}_{n_1,n_2}$ is $n_1 + n_2 - \gcd(n_1 + 1,n_2 + 1) + 1$.
This is because 
 $g_1 = 1 + \theta + \cdots + \theta^{n_1}$ and $g_2 = 1 + \theta + \cdots + \theta^{n_2}$.
The roots of $g_1$ and $g_2$ are respectively the $(n_1 + 1)$th and $(n_2 + 1)$th roots of unity not equal to 1.
It follows that there are $\gcd(n_1 + 1,n_2 + 1) - 1$ repeated roots of $g_1g_2$. 
\end{example}

\begin{example}\label{ex:sameN} 
If we choose the coefficients $c_{ij}$ to be the binomial coefficients $\binom{n_i}{j}$,  then $g_i=(1+\theta_{d-1})^{n_i}$ for all $i$.
For this choice of coefficients, the ML degree is 1 by Theorem~\ref{thm:MLDegRatScroll}, and the likelihood equations \eqref{likeqn1} have one solution given by 
$$\theta_{i}=
\frac{-u_{(d-1)+}u_{i+}}{u_{++}^{2} }
\left(1+\theta_{d-1}\right)^{n_{d-1}-n_{i}}\text{ for }i=1,2,\dots,d-2,
$$
where $u_{++} = \sum_{i}\sum_{j} u_{ij}$, and $\theta_{d-1}$ equals the unique solution to \[(Au)_{d-1}(1+\theta_{d-1}) =\theta_{d-1}( u_{(d-1)+} + \sum_{i=1}^{d-2} n_i u_{i+}).\]
\end{example}

\section{Veronese Embeddings}  \label{sec:ver}

In this section we study Veronese and Veronese-type varieties.

\begin{definition}
Consider the $(d-1)\times n$ integer matrix $A$ with columns that are the non-negative integer vectors whose coordinate sum $\leq k$. The projective toric variety defined by $A$ is the \textit{Veronese} ${\rm Ver}(d-1,k)$ for $d,k \geq 1$. 
\end{definition}

It can be seen that $\deg({\rm Ver}(d-1,k))= k^{d-1}$. A conjecture presented in \cite{V16} is that under the standard embedding, that is, with scaling vector $c=(1,1,\ldots,1,1)$, the ML degree of any Veronese variety equals its degree. We now show this is true when $k \leq d-1$.

 \begin{proposition} \label{prop:Ver(dk)}
 Consider ${\rm Ver}(d-1,k)$ for $k\leq d-1$ which is  embedded 
using the map \eqref{eq:scaledEmbedding} with scaling given by $c=(1,1,\ldots,1, 1)$. Then $c\notin \Sigma_A$ and $\mldeg( {\rm Ver}(d-1,k))=k^{d-1}$.
 \end{proposition}  
 \begin{proof}
First note that, by induction, it is sufficient to show that $c\notin \Delta_A$. 
To prove this we need to show that the hypersurface defined by the polynomial 
$$f=\sum \theta^{a_i}=1+\sum_{i=1}^k h_i$$
is nonsingular, where $h_i$ is the $i$th complete homogeneous symmetric polynomial in $\theta_1,\dots, \theta_{d-1}$.  
The complete homogeneous symmetric polynomials $h_1, \dots, h_{d-1}$ form an algebra basis for $\Lambda_{d-1}$, the ring of symmetric polynomials in variables $\theta_1,\dots, \theta_{d-1}$, see \cite[(2.8)]{IGM98}. The regular embedding $\varphi$ of $\Lambda_{d-1}$ into the polynomial ring $\bC[\theta_1,\dots, \theta_{d-1}]$ (specified by the definition of $h_i$ in terms of $\theta_j$'s) induces an embedding $\varphi^*$ of ${\rm Spec}(\Lambda_{d-1})$ into $\mathbb{A}^{d-1}={\rm Spec}(\bC[\theta_1,\dots, \theta_{d-1}])$. Consider the hypersurface $V=V(1+h_1+h_2+\cdots+h_k)\subset {\rm Spec}(\Lambda_{d-1})$. For $k\leq d-1$ we have that $1+h_1+h_2+\cdots+h_k$ is a linear form in the indeterminates  $h_1,\dots, h_{d-1}$. Hence $V$ is smooth in ${\rm Spec}(\Lambda_{d-1})$. Since the induced map $\varphi^*$ is a regular embedding, $\varphi^* (V)$ is smooth in $\mathbb{A}^{d-1}$, and by the definition of $h_i$ and the map $\varphi$ we have that $
 \varphi(1+\sum_{i=1}^k h_i)=f$.
 Thus $f$ defines a smooth hypersurface in $\mathbb{A}^{d-1}$, and $c=(1,1,\dots, 1,1)\notin \Delta_A$. The conclusion follows by Theorem \ref{mainthm}.
 \end{proof}

\begin{remark}[Hypersimplex]
Let $k \leq d-1$. Consider the $(d-1)\times n$ integer matrix $A$ where $n = \binom{d}{k}$ and where the columns of $A$ are the vectors in $\{0,1\}^{d-1}$ that have precisely
$k$ or $k-1$ entries equal to $1$. The $(d-1)$-dimensional polytope
$Q = {\rm conv}(A)$ is called the {\em hypersimplex}. The projective toric variety $V$ associated to $A$ represents generic torus orbits on the Grassmannian of $k$-dimensional linear subspaces in $\bC^d$. It is shown in the discussion preceding \cite[Proposition 4.7]{HS17} that $c=(1,1,\dots,1,1)\notin \Sigma_A$ for this matrix $A$. We remark that this result can be seen from the same proof given in Proposition \ref{prop:Ver(dk)}.

Let $e_{k-1}$ and $e_k$ be elementary symmetric 
polynomials in $\theta_1, \ldots, \theta_{d-1}$ of degree $k-1$ and $k$, respectively. Similar to above, we need to only show that the hypersurface defined by the polynomial 
\[ f = e_{k-1} + e_{k}
\]
is nonsingular. This proceeds in a manner identical to the proof of Proposition \ref{prop:Ver(dk)} since the
elementary symmetric polynomials also form an algebra basis for 
$\Lambda_{d-1}$, and hence are algebraically independent (see \cite[(2.4)]{IGM98}). Thus, for the projective toric variety
$\V$ associated to the hypersimplex, ${\rm mldeg}(V )$ is equal to the normalized volume of the hypersimplex $Q$, which is the Eulerian number $A(d - 1, k - 1)$.

It should be noted that the form of the definition of the hypersimplex given above differs slightly from that in \cite{HS17}, however the definitions are equivalent (i.e.~the resulting exponents of the monomial maps differ by elementary row operations). The difference is due to the fact that in \cite{HS17} a projective toric variety is defined by a monomial map given by a $d\times n$ integer matrix $\mathcal{A}$ with full rank and the vector $(1,\dots, 1)$ in the row space. In the convention used in our note we (equivalently) assume that the last row of what would be their matrix $\mathcal{A}$ consists only of ones and omit this row from our matrix $A$. 

\end{remark}

Returning to ${\rm Ver}(d-1,k)$, we note that in the case $k=2$, the polynomial
$f$ reduces to a quadratic form, and a direct proof that 
$c =(1,1, \ldots, 1,1) \notin \Sigma_A$ can be given by looking at the principal minors of the corresponding symmetric matrix  \begin{equation}
C=\begin{pmatrix}
 2c_{00} & c_{01} & \cdots & c_{0(d-1)}\\
  c_{01} & 2c_{11} & \cdots & c_{1(d-1)}\\
   \vdots & \vdots & \ddots & \vdots\\
    c_{0(d-1)} & c_{1(d-1)} & \cdots & 2c_{(d-1)(d-1)}
 \end{pmatrix},\label{eq:CmatVer(d2)}\end{equation} 
where $f = (1, \theta_1, \ldots, \theta_{d-1}) C  (1, \theta_1, \ldots, \theta_{d-1})^T$. 
 \begin{proposition}
Consider ${\rm Ver}(d-1,2)$. Then $c=(1,1,\dots, 1,1) \notin \Sigma_A$. In particular, we have that ${\rm mldeg} ( {\rm Ver}(d-1,2))=\deg( {\rm Ver}(d-1,2))=2^{d-1}.$
 \end{proposition}  
 \begin{proof}
 Let $Q=\conv(A)$. From \eqref{eq:PrincipalADetProduct} we have that $$\Sigma_A=\hspace{-3mm}\bigcup_{\Gamma \text{ face of } Q} \nabla_{A\cap \Gamma}.$$ The A-discriminant of the toric variety corresponding to each face will be given by a principal minor (of size corresponding to the dimension of the face) of the symmetric matrix $C$ in \eqref{eq:CmatVer(d2)}. Any principal minor of size $r\times r$ will have the same form as $C$.
 Hence, to show that $c=(1,1,\dots,1,1)\notin \Sigma_A$ we need only verify that the $r\times r$ matrix $$\tilde{C}=C(1,1,\dots,1,1)=\begin{pmatrix}
 2 & 1 & \cdots & 1\\
  1 & 2 & \cdots & \vdots\\
   \vdots & 1 & \ddots & 1\\
    1 & 1& \cdots & 2
 \end{pmatrix}$$ has a nonzero determinant. A calculation shows that $\det(\tilde{C})=r+1\neq 0$ for all $r\geq 2$. Thus, by Theorem~\ref{mainthm}, it follows that ${\rm mldeg} ( {\rm Ver}(d-1,2))=\deg( {\rm Ver}(d-1,2))=2^{d-1}.$
 \end{proof}
Now we give a sufficient condition for the opposite effect to happen: having ML degree equal to one, instead of having ML degree equal to the degree. The polynomial $f$ has terms of degree $\leq k$ 
in $\theta_1, \ldots, \theta_{d-1}$. It uniquely corresponds to
a symmetric tensor $\mathcal{C}$ of size $d \times d \times \ldots \times d$, where the product is $k$-fold. We say $\mathcal{C}$ has tensor rank one 
if the polynomial $f$ is a power of a linear form.
\begin{theorem}\label{Thmrank1}
Let $\mathcal{C}$ be the symmetric tensor corresponding to $f$ where
\[ f = \sum_{{\deg(\theta^a) \leq k } }  c_a \theta^a.
\]
If the tensor 
rank of $\mathcal{C}$ is equal to one then $\mldeg({\rm Ver}(d-1,k)^c)=1$. 
 \end{theorem}
 \begin{proof}
 In this case we have that $f = L^k$ where $L=b_0 + b_1\theta_1 + \cdots + b_{d-1} \theta_{d-1}$. Then in the likelihood equations \eqref{likeqn2} we have that $L \neq 0$ and
 \begin{equation}
 \begin{array}{ccc} 
(Au)_1 L^k  & = & (u_{+}k b_1) \theta_{1} L^{k-1} \\
(Au)_2 L^k & = & (u_{+} k b_2) \theta_{2} L^{k-1} \\
 & \vdots\\
(Au)_{d-1} L^k & = & ( u_{+} k b_{d-1} )\theta_{d-1}  L^{k-1}.
\end{array}
\end{equation}
So the factor $L^{k-1}$ cancels out and we obtain a linear system in the $\theta_i$'s. 
 \end{proof}

\begin{example}
Let $d=3$ and $k=2$, that is, we consider the second Veronese embedding of $\bP^2$. In this case the matrix defining ${\rm Ver}(2,2)$ is $
A =\begin{pmatrix}
0 & 1 & 2 & 0  & 0& 1 \\
0 & 0 & 0 & 1  & 2& 1
\end{pmatrix}
.$ The polytope $Q={\rm conv}(A)$ is given in Figure \ref{polytope} with the labeled lattice point $a_i$ corresponding to the $i$th column of the matrix $A$. 
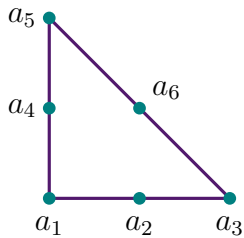
\begin{figure}[h!]

 \begin{center} \begin{tikzpicture}[scale=1.2]
\draw [pur1,very thick](0,0) -- (0,2);
\draw [pur1,very thick](0,0) -- (2,0);
\draw [pur1,very thick](2,0) -- (0,2);

\node at (0,-.3) {$a_1$};
\node at (1,-.3) {$a_2$};
\node at (2,-.3) {$a_3$};
\node at (-.3,1) {$a_4$};
\node at (-.3,2) {$a_5$};
\node at (1.3,1.2) {$a_6$};

\fill[teal] (0,0) circle[radius=2pt];
\fill[teal] (0,1) circle[radius=2pt];
\fill[teal] (1,0) circle[radius=2pt];
\fill[teal] (1,1) circle[radius=2pt];
\fill[teal] (2,0) circle[radius=2pt];
\fill[teal] (0,2) circle[radius=2pt];
\end{tikzpicture}
\end{center}\caption{The polytope $Q={\rm conv}(A)$ of ${\rm Ver}(2,2)$.} \label{polytope}
\end{figure}

By \eqref{eq:PrincipalADetProduct}, the ideal of $\Sigma_A \subset (\bC^*)^6$ is defined by the polynomial $E_A$ in \eqref{VerE_A}, \small
\begin{align} \label{VerE_A}
E_A&= \Delta_A \cdot \Delta_{[a_1\; a_4\; a_5]}\cdot \Delta_{[a_3\; a_5\; a_6]}\cdot \Delta_{[a_1\; a_2\; a_3]} \nonumber \\ 
&=\det \begin{pmatrix}
2c_{00} & {c_{10}} & {c_{01}} \\ {c_{10}} & 2c_{20} & {c_{11}} \\ {c_{01}} & {c_{11}} & 2c_{02}
\end{pmatrix}\det \begin{pmatrix}
2c_{00} & {c_{10}}  \\ {c_{10}} & 2c_{20} 
\end{pmatrix}\det\begin{pmatrix}
2c_{20} & {c_{11}}  \\ {c_{11}} & 2c_{02} 
\end{pmatrix}\det\begin{pmatrix}
2c_{00} & {c_{01}}  \\ {c_{01}} & 2c_{02} 
\end{pmatrix}.
\end{align}
\normalsize 

For generic values of $c_{ij}$ in $C$ as in equation~\eqref{eq:CmatVer(d2)}, including $c_{ij}=1$, we have that  $\mldeg(V^c) = 4=2^2$. On the other hand, if we take a scaling $c$ so that $C$ is the matrix with all entries equal to $2$, the ML degree drops to $\mldeg(V^c)=1$. See Table \ref{Vertable} for other combinations. Note that when the rank of $C$ is 1, the ML degree is 1, illustrating Theorem~\ref{Thmrank1}. We also see that the converse does not hold: the ML degree can be 1 even with the rank of $C$ equal to 3.    

\begin{table}[h!]
\centering
\resizebox{.7\linewidth}{!}{
\begin{tabular}{@{} l *8c @{}}
\toprule 
 \multicolumn{1}{c}{{\color{Ftitle} $C$}}    & {\color{Ftitle} $\Delta_A$}  &   {\color{Ftitle} $\Delta_{[a_1\; a_4\; a_5]}$ } &   {\color{Ftitle} $\Delta_{[a_3\; a_5\; a_6]}$ } &   {\color{Ftitle} $\Delta_{[a_1\; a_2\; a_2]}$ } &   {\color{line} ${\rm mldeg}$ }  \\ 
 \midrule 
   {$\begin{bmatrix} 2 & 1 & 1 \\ 1 & 2 & 1 \\ 1 & 1 & 2 \end{bmatrix}$}  & $\neq0$ & $\neq0$& $\neq0$& $\neq0$& \textbf{\color{line} 4}\vspace{2mm}\\
  $\begin{bmatrix} 2 & 2 & 1 \\ 2 & 2 & 3 \\ 1 & 3 & 2 \end{bmatrix}$  & 0 & 0& $\neq0$& $\neq0$& \textbf{\color{line} 3}\vspace{2mm}\\ 
     $\begin{bmatrix} 2 & 2 & 1 \\ 2 & 2 & 2 \\ 1 & 2 & 2 \end{bmatrix}$ & 0 & 0 & 0& $\neq0$& \textbf{\color{line} 2}\vspace{2mm}\\ 
   {\small$\begin{bmatrix} \text{-}2 & 2 & 2 \\ 2 & \text{-}2 & 2 \\ 2 & 2 & \text{-}2 \end{bmatrix}$}  & $\neq0$ & 0& 0& 0 & \textbf{\color{line} 1}\vspace{2mm}\\ 
   {\small$\begin{bmatrix} 17 & 22 & 27 \\ 22 & 29 & 36 \\ 27 & 36 & 45 \end{bmatrix}$} & 0 & $\neq0$ & $\neq0$ & $\neq0$& \textbf{\color{line} 3}\vspace{2mm}\\ 
  $\begin{bmatrix} 2 & 3 & 3 \\ 3 & 5 & 5 \\ 3 & 5 & 5 \end{bmatrix}$ & 0 & $\neq0$ & 0& $\neq0$& \textbf{\color{line} 2}\vspace{2mm}\\ 
  $\begin{bmatrix} 2 & 2 & 2 \\ 2 & 2 & 2 \\ 2 & 2 & 2 \end{bmatrix}$ & 0 & 0 & 0& 0 & \textbf{\color{line} 1}\vspace{2mm}\\ 
\bottomrule
 \end{tabular}}\vspace{1mm}
\caption{The ML degree of $({\rm Ver}(2,2)^c)$ for different scalings $c_{ij}$ in the matrix $C$.} \label{Vertable}
 \end{table}
 
\end{example}

\section{Segre Embeddings}\label{sec:segre}

In this section, we study the maximum likelihood degree of the scaled Segre embedding $V^c$. 
We give a sufficient condition for $V^c$ to have ML degree one.

For $c\in (\bC^*)^{mn}$, let $\psi^c$ be the map defined by 
\begin{equation}
  \psi^c\left(s,\theta^{(1)}_1, \ldots, \theta^{(1)}_m, \theta^{(2)}_1, \ldots, \theta^{(2)}_n\right) = \left(c_{11} s \theta^{(1)}_1 \theta^{(2)}_1, \ldots, c_{ij} s \theta^{(1)}_i \theta^{(2)}_j, \ldots, c_{mn} s \theta^{(1)}_m \theta^{(2)}_n\right)
\end{equation}
Then $V^c$ is a scaled Segre embedding of $\mathbb{P}^{m-1}\times \mathbb{P}^{n-1}$. In terms of the coordinates $p_{ij}$ with $1 \leq i \leq m$ and $1 \leq j\leq n$ on $\mathbb{P}^{mn-1}$, the defining ideal $I^c$ of $V^c$ is given by the 2-minors of the $m\times n$ matrix with $ij$th entry $\frac{p_{ij}}{c_{ij}}$. 

From this it follows that the likelihood equations (\ref{eqlik3}) are
\begin{equation*}
  \frac{u_{i+}}{u_{++}} = \frac{p_{i+}}{p_{++}} \,\text{ for }\, 1 \leq i \leq m,  
  \quad
  \frac{u_{+j}}{u_{++}} = \frac{p_{+j}}{p_{++}} \,\text{ for }\, 1 \leq j \leq n, \quad
  \frac{p_{ij}p_{kl}}{c_{ij}c_{kl}} = \frac{p_{il}p_{kj}}{c_{il}c_{kj}}.
\end{equation*}
Here $u_{i+} = \sum_{j=1}^n u_{ij}$, $u_{+j} = \sum_{i=1}^m u_{ij}$, and $u_{++} = \sum_{i=1}^m \sum_{j=1}^n u_{ij}$. We define $p_{i+}$, $p_{+j}$, and $p_{++}$ analogously.

For all $c\in(\bC^*)^{mn}$, we have  $\deg(V^c) =  \binom{m+n-2}{m-1}$. 
If $c = (1,1,\ldots,1,1)$, then the ML degree of $V^{c}$ is well known to be one (see~\cite[Example 1.12]{PS05} and~\cite[Example 2.1.2]{DSS09}). 
Also the principal $A$-determinant for Segre embedding is easy to describe \cite[Chapter 9, Section 1]{GKZ94}:
\begin{equation}
 E_A = \prod_{i=1}^{\min(m,n)} \prod_{{1 \leq a_1 < \cdots < a_i \leq m}\atop{ 1 \leq b_1 < \cdots < b_i \leq n}} \mathrm{det}[a_1, \ldots, a_i; b_1, \ldots, b_i]\label{eq:EASegre}
\end{equation}
where $[a_1, \ldots, a_i; b_1,\ldots, b_i]$ is the submatrix of $c$ indexed by the rows $1 \leq a_1 < \cdots < a_i \leq m$ and 
the columns $1 \leq b_1 < \cdots < b_i \leq n$. 
 
\begin{example}\label{3m3n} Let $V^c$ be the scaled Segre variety with scaling
\begin{equation*}
c=
\left( \begin{array}{ccc}
1 & c_{12} & c_{13} \\
c_{21} & c_{22} & 1 \\
1 & 1 & 1 \end{array} \right) \in M_{3,3}(\mathbb{C}^*).
\end{equation*}
Here, $m=n=3$, and the matrix $A$ is given by
\begin{equation*}
A=
\left( \begin{array}{ccccccccc}
1 & 1 & 1 & 0 & 0 & 0 & 0 & 0 & 0 \\
0 & 0 & 0 & 1 & 1 & 1 & 0 & 0 & 0 \\
0 & 0 & 0 & 0 & 0 & 0 & 1 & 1 & 1 \\
1 & 0 & 0 & 1 & 0 & 0 & 1 & 0 & 0 \\
0 & 1 & 0 & 0 & 1 & 0 & 0 & 1 & 0 \\
0 & 0 & 1 & 0 & 0 & 1 & 0 & 0 & 1 
\end{array} \right).
\end{equation*}

The degree of $V^c$ is $\binom{3+3-2}{3-1}=6$. The assignments
\begin{equation*}
(c_{12},c_{13},c_{21},c_{22}) \in \{(1,1,1,1),\ (2,1,1,1),\ (2,3,1,1),\ (2,3,1,2),\ (2,3,2,1),\ (2,3,2,3)\}
\end{equation*}
produce, respectively, the ML degrees $\mldeg(V^c)=1,2,3,4,5,6$. 
This shows $V^c$ can have any ML degree $i$, for $1 \leq i \leq \operatorname{deg}(V^c)$. 
\end{example}

\begin{conjecture} For each $i \in \{1,\ldots,\deg(\V)\}$, there exists $c \in (\bC^*)^{mn}$ such that the scaled Segre variety $V^c$ has ML degree $i$.
\end{conjecture}

\begin{example}\label{4waySegre}
The scaled toric variety $V^c$ with the configuration
\begin{equation*} 
A=\left( \begin{array}{cccccccccccccccc}
1 & 1 & 1 & 1 & 1 & 1 & 1 & 1 & 0 & 0 & 0 & 0 & 0 & 0 & 0 & 0 \\
1 & 1 & 1 & 1 & 0 & 0 & 0 & 0 & 1 & 1 & 1 & 1 & 0 & 0 & 0 & 0 \\
1 & 1 & 0 & 0 & 1 & 1 & 0 & 0 & 1 & 1 & 0 & 0 & 1 & 1 & 0 & 0 \\
1 & 0 & 1 & 0 & 1 & 0 & 1 & 0 & 1 & 0 & 1 & 0 & 1 & 0 & 1 & 0 \end{array} \right)
\end{equation*}
is isomorphic to the scaled Segre embedding of $\mathbb{P}^1 \times \mathbb{P}^1 \times \mathbb{P}^1 \times \mathbb{P}^1$ into $\mathbb{P}^{15}$. The degree of $V^c$ in this instance is $24$. For each of the possible ML degrees $i=1,2,\ldots,24$, we have found a scaling vector $c \in (\bC^*)^{16}$ for which $\textrm{mldeg}(V^c)=i$. 
We see from \eqref{eq:EASegre} that for a point to lay in the vanishing of the principal A-determinant we must have that certain minors of the matrix $\left( c_{ij}\right)$ vanish; this observation informed our choice of scaling vector $c$.
\end{example}

We conclude this section with a result analogous to Theorem~\ref{Thmrank1}.
\begin{proposition}\label{rank1}
If $\mathrm{rank}(c)=1$, then $\mathrm{mldeg}(V^c)=1$.
\end{proposition}
\begin{proof} If $\operatorname{rank}(c) = 1$ then all the 2-minors of $c$ vanish. This implies that $c_{ij}c_{kl} = c_{il}c_{kj}$ for all $1 \leq i, k \leq m$ and $1 \leq j,l \leq n$. Then our likelihood equations reduce to 
\begin{equation*}
  \frac{u_{i+}}{u_{++}} = \frac{p_{i+}}{1} \,\text{ for }\, 1 \leq i \leq m 
  \quad
  \frac{u_{+j}}{u_{++}} = \frac{p_{+j}}{1} \,\text{ for }\, 1 \leq j \leq n, \quad
  {p_{ij}p_{kl}} = {p_{il}p_{kj}},
\end{equation*}
which are the likelihood equations for the unscaled Segre embedding of $\mathbb{P}^{m-1} \times \mathbb{P}^{n-1}$. This is known to have ML degree equal to one. 
\end{proof}

\section{Hierarchical log-linear and graphical models}\label{sec:graphical}

A large class of discrete exponential models that are used in statistical practice are hierarchical log-linear models \cite{BFH07} and undirected graphical models \cite{Lau96}.
In both cases, the toric variety is constructed via 
a parametrization based on a simplicial
complex where the vertices of the simplicial complex correspond to discrete random variables. 
In the second case, the simplicial complex is the clique complex of
a given graph. 
These toric varieties, their defining ideals, and the polyhedral
geometry of  $Q = \conv(A)$ of the underlying matrix $A$ have been
intensely studied starting with \cite{HS02} for the hierarchical log-linear
models and \cite{GMS06} for graphical models. 

Decomposable models form a subclass of these models, and they
present attractive factorization properties. For instance, a graphical
model is decomposable if and only if the underlying graph is a chordal
graph \cite{Lau96}.
In particular, 
we have the following. 
\begin{theorem}\cite[Theorem 4.4]{GMS06} A hierarchical log-linear model, and
hence an undirected graphical model, is decomposable if and only if the ML
degree of the corresponding toric variety is one.
\end{theorem}
For example, the (unscaled) Segre embedding of $\bP^{m_1} \times \cdots \times \bP^{m_k}$
is a decomposable graphical model, and hence has ML degree one. In this section we consider some examples of nondecomposable models. First we study a family of examples known as no-three-way interaction models.
\subsection{No-three-way interaction models}
\begin{example}[Binary 3-cycle] Consider the hierarchical log-linear model arising from the $3$-cycle depicted in 
Figure~\ref{fig:binary-3-cycle}.  This is also known as the no-three-way interaction model. We let the random 
variables $X,Y,Z$ corresponding to the vertices to be binary variables.
One parametrization $\psi:(\bC^*)^6 \to (\bC^*)^8$ is given by the matrix \[
A = \left( \begin{array} {c c c c c c c c}
0 & 0 & 0 & 0 & 1 & 1 & 1 & 1 \\
0 & 0 & 1 & 1 & 0 & 0 & 1 & 1 \\
0 & 1 & 0 & 1 & 0 & 1 & 0 & 1 \\
0 & 0 & 0 & 0 & 0 & 0 & 1 & 1 \\
0 & 0 & 0 & 1 & 0 & 0 & 0 & 1 \\
0 & 0 & 0 & 0 & 0 & 1 & 0 & 1 
\end{array}\right).
\]
We see that the $6 \times 8$ matrix $A$ is in general position
as in Section \ref{sec:hypersurface}. Hence, the toric variety $\V$ defined by $A$ is a hypersurface with a generator of full support. 
For $i,j,k \in \{0,1\}$, let $p_{ijk} = \mathrm{prob}(X = i, Y = j, Z = k)$.  Labeling the columns of the matrix with the $p_{ijk}$ in increasing lexicographic order, we find that the unique toric generator is \[g = p_{000} p_{011} p_{101} p_{110} - p_{001}p_{010}p_{100}p_{111}.\] By Corollary \ref{cor-hypersurface},  the principal $A$-determinant is equal to the $A$-discriminant, and it is \[\Delta_A = c_{000} c_{011} c_{101} c_{110} - c_{001}c_{010}c_{100}c_{111}.\]  

The toric variety itself has degree $4$ since it has a single degree $4$ generator. However, if we choose the scaling coefficients $c_{ijk}=1$ for all $i,j,k \in \{0,1\}$, we see that this scaling vector lies in $\LocusEAf$, and thus $\mldeg(\V^c)$ drops, in this case, to $3$.
In fact, for this example, the only possibilities for $\mldeg(\V^c)$ are $3$ when the scaling vector lies in $\LocusEAf$ and $4$ when the scaling vector does not.

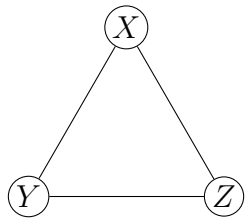
\begin{figure}[!h]
\centering
\begin{tikzpicture}[scale=.5]
\def \radius {3cm}
\def \n {3}
\foreach \v in {1,...,\n} {
\node(v\v) at ({360/\n*\v - 30}:\radius) {};
}
\node[circle,inner sep = 1pt,draw](X) at (v1) {$X$};
\node[circle,inner sep = 1pt,draw](Y) at (v2) {$Y$};
\node[circle,inner sep = 1pt,draw](Z) at (v3) {$Z$};
\draw (X) -- (Y) -- (Z) -- (X);
\end{tikzpicture}
\caption{The 3-cycle.}\label{fig:binary-3-cycle}
\end{figure}
\end{example}

\begin{proposition} The ML degree of the binary $3$-cycle is $4$ unless $c \in (\bC^*)^{d+1}$ is 
in $\LocusEAf$. If $c \in \LocusEAf$, then $\mldeg(\V^c) = 3$. 
\end{proposition}

\begin{proof} We fix the lexicographic monomial ordering \[p_{000} > p_{001} > p_{010} > \cdots > p_{111} > c_{000} > c_{001} > \cdots > c_{111}\] over $\bQ[p_{000},p_{001},p_{010},p_{011},p_{100},p_{101},p_{110},p_{111},c_{000},c_{001},c_{010},c_{011},c_{100},c_{101},c_{110},c_{111}]$. 

Let $u = (u_{000},\ldots,u_{111})$ be a data vector in $\bZ^8$ and let $u_+ = \sum u_{ijk}$. Consider \[I = \left\langle\,  g \, , \, \Delta_A \, , \, L_f(p) - 1 \, , \,  u_+L_1(p) - (Au)_1 \, , \, \ldots \, , \, u_+ L_6(p) - (Au)_6 \, \right\rangle
\]
where $f = \sum c_{ijk} p_{ijk}$.
Using \cite{M2} and the random data vector $u=(2,3,5,7,11,13,17,19)$, we calculate a Gr\"{o}bner basis for $I$:
\begin{align*}
g_1 &= c_{000} c_{011} c_{101} c_{110}-c_{001} c_{010} c_{100} c_{111} \\
g_2 &= 5021863 p_{111}^{3} c_{111}^{3}-3752210 p_{111}^{2} c_{111}^{2}+ 984280 p_{111} c_{111}-89856 \\
g_3 &= 77 p_{110} c_{110}+77 p_{111} c_{111}-36 \\
g_4 &= 17472 p_{110} c_{001} c_{010} c_{100}
		-456533 p_{111}^{3} c_{000} c_{011} c_{101} c_{111}^{2}
		+341110 p_{111}^{2} c_{000} c_{011} c_{101} c_{111} + \cdots \\
g_5 &= 77 p_{101} c_{101}+77 p_{111} c_{111}-32 \\
g_6 &= 19656 p_{101} c_{001} c_{010} c_{100}
	-456533 p_{111}^{3} c_{000} c_{011} c_{110} c_{111}^{2}
	+341110 p_{111}^{2} c_{000} c_{011} c_{110} c_{111} + \cdots \\
g_7 &= 77 p_{100} c_{100}-77p_{111} c_{111}+8 \\
g_8 &= 184320 p_{100} c_{000} c_{011}
	-73501813 p_{101} p_{110} p_{111}^{2} c_{001} c_{010} c_{111}^{3}
	+ \cdots \\
g_9 &= 77 p_{011} c_{011}+77 p_{111} c_{111}-26 \\
g_{10} &= 359424 p_{011} c_{001} c_{010}
		-115502849 p_{100} p_{111}^{3} c_{000} c_{101} c_{110} c_{111}^{2} + \cdots \\
g_{11} &= 11 p_{010} c_{010}-11 p_{111} c_{111}+2 \\
g_{12} &= 287539200 p_{010} c_{000}
		-1097794317389 p_{011} p_{101} p_{110} p_{111}^{2} c_{001} c_{100} c_{111}^{3} + \cdots \\
g_{13} &= 77 p_{001} c_{001}-77 p_{111} c_{111}+16 \\
g_{14} &= 4313088 p_{001} c_{000}
		-12619941719 p_{011} p_{101} p_{110} p_{111}^{2} c_{010} c_{100} c_{111}^{3} + \cdots \\
g_{15} &= 1725235200 p_{000}
		-59243384844259 p_{001} p_{010} p_{100} p_{111}^{3} c_{011} c_{101} c_{110} c_{111}^{2} + \cdots \\
\end{align*}
 As long as each $c_{ijk} \in \bC^*$ and satisfies the equation $g_1 = \Delta_A = 0$, we see that $g_2$ is a univariate polynomial in $p_{111}$ of degree $3$, and the initial terms of $g_3$ through $g_{15}$ have degree $1$ in $p_{ijk}$.  Therefore, when the coefficient vector lies in $\LocusEAf$, the ML degree is always $3$.
\end{proof}

More generally, we consider the no-three-way interaction model $C_m$ based on the $3$-cycle with one $m$-ary variable and two binary variables. That is, for $i \in \{0,\ldots, m-1\}$ and $j,k \in \{0,1\}$, we set 
\[
p_{ijk} = \mathrm{prob}(X = i, Y = j, Z = k) = a_{ij}b_{ik}c_{jk}.  
\]
One can compute the normalized volume of $Q = \conv(A_m)$ for the  matrix $A_m$ 
associated $C_m$. This gives a formula for the degree of $C_m$.
\begin{proposition} The degree of $C_m$ is $m2^{m-1}$.
\end{proposition}

\noindent When $m=3$, $\deg(C_3) = 12$ and $\mldeg(C_3) = 7$. The toric ideal is
\begin{small}
\[\langle  p_{100}  p_{111}  p_{201}  p_{210} - p_{101}  p_{110}  p_{200}  p_{211}, \, 
p_{000}  p_{011}  p_{101}  p_{110} - p_{001}  p_{010}  p_{100}  p_{111},\, 
p_{000}  p_{011}  p_{201}  p_{210} - p_{001}  p_{010}  p_{200}  p_{211}  \rangle.
\]
\end{small}
Observe that the toric variety is not a hypersurface, so calculating the principal $A$-determinant requires more work than the previous example.  
When we compute $\Delta_{A_3}$, the $A$-discriminant for the 
whole polytope $Q = \conv(A_3)$, we do {\it not} get a hypersurface: 

\begin{align*} \Delta_A = \langle &c_{100}  c_{111}  c_{201}  c_{210} - c_{101}  c_{110}  c_{200}  c_{211},\\ 
&c_{000}  c_{011}  c_{101}  c_{110} - c_{001}  c_{010}  c_{100}  c_{111},\\ 
&c_{000}  c_{011}  c_{201}  c_{210} - c_{001}  c_{010}  c_{200}  c_{211} \rangle .
\end{align*}
We already see that $c = (1,1, \ldots, 1,1)$ is in $\LocusEAf$ and 
hence we know $\mldeg(C_3^c) < \deg(C_3) = 12$. 
Also comparing the toric variety and the discriminant locus $\nabla_{A_3}$
we see that they are identical. Such a toric variety is known as self-dual. Based on our computations for $m \leq 5$ we state the following
conjecture. 
\begin{conjecture} The no-three-way interaction model 
with one $m$-ary variable and two binary variables is self-dual.
\end{conjecture}
To compute the entire principal $A$-determinant, we need to consider
all faces of $Q$. If a face $\Gamma$ has dimension $e$ 
and $|\Gamma \cap A_3| = e+1$, then $\Delta_{\Gamma \cap A_3}$ is
the unit ideal. Moreover, for all codimension one and codimension 
two faces $\Gamma$ which are not simplices, we observe that 
$\nabla_{\Gamma \cap A_3}$ are not hypersurfaces. They also lie 
in a coordinate hyperplane. For instance, for the facet $\Gamma$ 
where the elements of $\Gamma \cap A_3$ correspond to 
\[
\{ p_{001}, p_{011}, p_{100}, p_{101}, p_{110}, p_{111}, p_{200}, p_{201}, p_{210}, p_{211}  \}
\]
the discriminant is defined by 
\[
\langle \, c_{001}, \, c_{011}, \, c_{100}c_{111}c_{201}c_{210} - c_{101}c_{110}c_{200}c_{211} \, \rangle.
\]
Therefore, for $c=(1,1, \ldots, 1,1)$,
these discriminants do not contribute to a drop in the ML degree. 
There is a total of three codimension $3$ faces that are not simplices.
In each case $\Gamma \cap A_3$ is in general position 
and their discriminants are given by Corollary~\ref{cor-hypersurface}. 
We list the indeterminates which correspond to the elements
of $\Gamma \cap A_3$ in each of these three faces:
\begin{align*}
&\{p_{000}, p_{001}, p_{010},p_{011},p_{100},p_{101},p_{110},p_{111}\},\\
&\{p_{000}, p_{001}, p_{010},p_{011},p_{200},p_{201},p_{210},p_{211}\},\\
&\{p_{100}, p_{101}, p_{110},p_{111},p_{200},p_{201},p_{210},p_{211}\}.
\end{align*}
It turns out that each face of codimension $>3$ is a simplex and 
there are no more discriminants contributing to the principal $A$-determinant. 
Based on our computations for $m \leq 5$, we state the following conjecture.
\begin{conjecture} The ML degree of $C_m$ is $2^m -1$.\end{conjecture}

\subsection{The binary $4$-cycle}
The binary $4$-cycle is the model we have used in Example~\ref{ex:binary-4-cycle}. Let $S, B, H, L$ be four binary random variables and let $X = (S, B, H, L)$ be the joint random variable where we 
set 
$$ p_{ijk\ell} \, \, = \, \, \mathrm{prob}(S = i, B=j, H = k, L =\ell).$$
The family of probability distributions for $X$ which factor according
to the graphical model depicted in Figure~\ref{fig:binary-4-cycle} can be described by the following monomial parametrization: let $a_{ij}, b_{jk}, c_{k\ell}, d_{i \ell}$
be parameters for $i,j,k, \ell \in \left\lbrace 0, 1\right\rbrace$ and let $p_{ijk\ell} = a_{ij}b_{jk}c_{k\ell} d_{i \ell}$. Below is the matrix $A$ arising from 
a different parametrization that gives a full rank matrix. 
The columns are labeled by $p_{0000}, p_{0001}, \ldots, p_{1111}$ in increasing lexicographic ordering.
\[ A = \left( 
\begin{array}{c c c c c c c c c c c c c c c c}
0 & 0 & 0 & 0 & 0 & 0 & 0 & 0 & 1 & 1 & 1 & 1 & 1 & 1 & 1 & 1\\
0 & 0 & 0 & 0 & 1 & 1 & 1 & 1 & 0 & 0 & 0 & 0 & 1 & 1 & 1 & 1\\
0 & 0 & 1 & 1 & 0 & 0 & 1 & 1 & 0 & 0 & 1 & 1 & 0 & 0 & 1 & 1\\
0 & 1 & 0 & 1 & 0 & 1 & 0 & 1 & 0 & 1 & 0 & 1 & 0 & 1 & 0 & 1\\
0 & 0 & 0 & 0 & 0 & 0 & 0 & 0 & 0 & 0 & 0 & 0 & 1 & 1 & 1 & 1\\
0 & 0 & 0 & 0 & 0 & 0 & 1 & 1 & 0 & 0 & 0 & 0 & 0 & 0 & 1 & 1\\
0 & 0 & 0 & 1 & 0 & 0 & 0 & 1 & 0 & 0 & 0 & 1 & 0 & 0 & 0 & 1\\
0 & 0 & 0 & 0 & 0 & 0 & 0 & 0 & 0 & 1 & 0 & 1 & 0 & 1 & 0 & 1
\end{array}
\right)
\]

\noindent
The Zariski closure $V$ of the image of this parametrization is a toric variety  that is defined by the following prime ideal: 
\begin{small}
 \begin{align*}
I        =  \langle \, & p_{1011}p_{1110} - p_{1010}p_{1111}, \,\, p_{0111}p_{1101} - p_{0101}p_{1111}, \,\, p_{1001}p_{1100} - p_{1000}p_{1101}, \,\, p_{0110}p_{1100} - p_{0100}p_{1110}, \\
      &  p_{0011}p_{1001} - p_{0001}p_{1011}, \,\, p_{0011}p_{0110} - p_{0010}p_{0111}, \,\, p_{0001}p_{0100} - p_{0000}p_{0101}, \,\, p_{0010}p_{1000} - p_{0000}p_{1010}, \\
      &  p_{0100}p_{0111}p_{1001}p_{1010} \, - \, p_{0101}p_{0110}p_{1000}p_{1011},  \,\, p_{0010}p_{0101}p_{1011}p_{1100} \, - \, p_{0011}p_{0100}p_{1010}p_{1101},  \\
      &  p_{0001}p_{0110}p_{1010}p_{1101} \, - \, p_{0010}p_{0101}p_{1001}p_{1110},  \,\, p_{0001}p_{0111}p_{1010}p_{1100} \, - \, p_{0011}p_{0101}p_{1000}p_{1110},  \\
      &  p_{0000}p_{0011}p_{1101}p_{1110} \, - \, p_{0001}p_{0010}p_{1100}p_{1111},  \,\, p_{0000}p_{0111}p_{1001}p_{1110} \, - \, p_{0001}p_{0110}p_{1000}p_{1111},  \\
      &  p_{0000}p_{0111}p_{1011}p_{1100} \, - \, p_{0011}p_{0100}p_{1000}p_{1111},  \,\, p_{0000}p_{0110}p_{1011}p_{1101} \, - \, p_{0010}p_{0100}p_{1001}p_{1111} \, \rangle.
\end{align*}
\end{small}

The degree of this toric variety is $64$ and $\mldeg(V) = 13$. This was computed in \cite[p.~1484]{GMS06} where the question of explaining the fact $\mldeg(V) = 13$
was first raised.  
One road to an explanation is to compute the $A$-discriminants
of all the faces of $Q=\conv(A)$, and then determine the contribution
of each for the drop in the ML degree. There are multiple faces 
that might contribute to such a drop. At this moment, we do 
not understand how different discriminants {\it interact}. Also, it is 
not possible to compute every discriminant. For instance, 
with standard elimination methods we were not able to compute 
$\Delta_A$ corresponding to the entire polytope. However, one can 
check that the polynomial $f = \sum_{i}^{16} \theta^{a_i}$ where $c = (1,1,\ldots, 1,1)$ has singularities in $(\bC^*)^{d-1}$. Hence, we conclude that $\Delta_A$ does contribute to the drop from the generic
ML degree.  However, we can compute discriminants for facets and lower codimension faces.

The polytope $Q$ has $24$ facets.  Of these, $6$ correspond to the $0$ ideal and $2$ correspond to toric hypersurfaces of the form $c_{0001}c_{1101} - c_{0110}c_{1100}$.
Of the remaining 16 facets, eight lie on a coordinate hyperplane. 
An example of such a discriminant is 
 \[\langle c_{1101}, c_{0011}, c_{0001}c_{0110}c_{1010}c_{1100} - c_{0010}c_{0101}c_{1000}c_{1110} - c_{0001}c_{0100}c_{1010}c_{1110} + c_{0000}c_{0101}c_{1010}c_{1110} \rangle.\]
There are 5 facets which give rise to hypersurfaces, and most are too long to display.  Of these discriminants, there are three of degree $5$, one of degree $6$, and one of degree $16$. The remaining three facets are neither of the above types.  Two have 5 generators of degree 4 of the form 
\begin{align*}
\langle &c_{0000}c_{0111}c_{1011}c_{1101} - c_{0011}c_{0100}c_{1001}c_{1111} + c_{0001}c_{0100}c_{1011}c_{1111} - c_{0000}c_{0101}c_{1011}c_{1111},\\ &c_{0000}c_{0111}c_{1011}c_{1100} - c_{0011}c_{0100}c_{1000}c_{1111},\\ &c_{0011}c_{0111}c_{1001}c_{1100} - c_{0001}c_{0111}c_{1011}c_{1100} - c_{0011}c_{0111}c_{1000}c_{1101} + c_{0011}c_{0101}c_{1000}c_{1111},\\ &c_{0000}c_{0111}c_{1001}c_{1100} - c_{0000}c_{0111}c_{1000}c_{1101} - c_{0001}c_{0100}c_{1000}c_{1111} + c_{0000}c_{0101}c_{1000}c_{1111},\\ &c_{0011}c_{0100}c_{1001}c_{1100} - c_{0001}c_{0100}c_{1011}c_{1100} + c_{0000}c_{0101}c_{1011}c_{1100} - c_{0011}c_{0100}c_{1000}c_{1101}\rangle,
\end{align*} and the remaining discriminant has $3$ generators of degree $2$,
\[
\langle c_{1001}c_{1100} - c_{1000}c_{1101}, c_{0011}c_{0110} - c_{0010}c_{0111}, c_{0001}c_{0100} - c_{0000}c_{0101}\rangle.
\]

Because there are discriminants of codimension one faces that are not hypersurfaces, we must also analyze the 168 codimension two faces of $Q$. Of these, 86 are trivial and 61 correspond to toric hypersurfaces. Of the remaining 21, there is only a single face whose discriminant does not lie on a coordinate hyperplane, and the discriminant of this face is a hypersurface generated by
\[
c_{0110}c_{1000}c_{1011}c_{1101} + c_{0100}c_{1001}c_{1011}c_{1110} - c_{0100}c_{1001}c_{1010}c_{1111}.
\]



\section{ML Estimate Homotopies}\label{sec:homotopy}

\newcommand{\cStat}{c_{stat}}
\newcommand{\pStat}{\expandafter\hat p_{stat}}
\newcommand{\FStat}{F_{stat}}
\newcommand{\thetaStat}{\expandafter\hat \theta_{stat}}

\newcommand{\cEasy}{c_{easy}}
\newcommand{\pEasy}{\expandafter\hat p_{easy}}
\newcommand{\FEasy}{F_{easy}}

In this section 
we use homotopy continuation to track between ML estimates of different scalings of a given statistical model. 
 Moreover, the endpoints may correspond to different scalings of the model with different ML degrees. 

Consider the case where we have found a particular scaling $\cEasy$ such that the ML degree drops to one, but we really wish to compute the ML estimate for the natural 
statistical model corresponding to the scaling $\cStat$. 
The strategy is to apply a parameter homotopy {\cite{Li89}, \cite{SomWamp89}} between $\cEasy$ and $\cStat$, tracking the unique solution of the former to a solution of the latter. 
We argue that the endpoint is the unique ML estimate in Birch's Theorem \ref{thm-Birch}.

  \begin{example}(Veronese)\label{ex:veronese}
 We illustrate the strategy with a Veronese model. 
Let $$A = \left( \begin{array}{c c c c c c} 1&1&1&1 & 1 & 1 \\ 0&1&2&0 & 1 & 0 \\ 0&0&0&1 & 1 & 2 \end{array}\right),$$ with $\cStat=(1,1,1,1,1,1)$ and $\cEasy=(1,2,1,2,2,1)$. 
We can check (see Section \ref{sec:ver}) that the ML degree corresponding to $\cStat$ is 4, while the ML degree for $\cEasy$ is 1. 

Suppose we observe the data vector $u=(1,3,5,7,9,2)$. 
Computing the unique solution for $\cEasy$ we obtain the ML estimate   
$\hat{\theta}_{easy}=(0.0493, 1.8333, 1.6667)$.  
We track this point with a parameter homotopy \eqref{eq:wroteOutHomotopy} towards $\cStat$, and obtain the
point $\hat{\theta}_{track}=(0.0863, 1.6326, 1.5150)$ with corresponding
$\pStat = (0.09, 0.14, 0.23, 0.13, 0.21, 0.20)$. 
The homotopy we track is
\begin{equation}
\begin{array}{r}
H\left(\theta,t\right) = t  \cdot \left(
\begin{array}{l}
27\theta_1\theta_2^2+54\theta_1\theta_2\theta_3+27\theta_1\theta_3^2+54\theta_1\theta_2+54\theta_1\theta_3+27\theta_1-27 \\
54\theta_1\theta_2^2+54\theta_1\theta_2\theta_3+54\theta_1\theta_2-22 \\
54\theta_1\theta_2\theta_3+54\theta_1\theta_3^2+54\theta_1\theta_3-20
\end{array}
\right) \\ + (1-t) \cdot\left(
\begin{array}{l}
27\theta_1\theta_2^2+27\theta_1\theta_2\theta_3+27\theta_1\theta_3^2+27\theta_1\theta_2+27\theta_1\theta_3+27\theta_1-27 \\
54\theta_1\theta_2^2+27\theta_1\theta_2\theta_3+27\theta_1\theta_2-22 \\
27\theta_1\theta_2\theta_3+54\theta_1\theta_3^2+27\theta_1\theta_3-20
\end{array}
\right)
\end{array}
\end{equation}
Or alternatively, if we define $c(t)=27+27t$, we can write it as
\begin{equation}\label{eq:wroteOutHomotopy}
\begin{array}{ccc}
H\left(\theta,t\right) 
&= & \left\{
\begin{array}{l}
27\theta_1\theta_2^2 + c(t)\theta_1\theta_2\theta_3 + 27\theta_1\theta_3^2 + c(t)\theta_1\theta_2 + c(t)\theta_1\theta_3 + 27\theta_1-27 \\
54\theta_1\theta_2^2 + c(t)\theta_1\theta_2\theta_3 + c(t)\theta_1\theta_2-22 \\
c(t)\theta_1\theta_2\theta_3 + 54\theta_1\theta_3^2 + c(t)\theta_1\theta_3-20
\end{array}\right. .
\end{array}
\end{equation}

To verify this,  we solve the critical equations for $\cStat$ to obtain the four solutions
$$
\begin{array}{rr}
(0.2888, 1.4316, -1.8931), &	 (0.3039, -1.8847, 1.3470),\\ 
  (0.8578, -0.7629, -0.7189), &  (0.0863, 1.6326, 1.5150),
\end{array}
$$
where  $\thetaStat$ is the solution with positive coordinates. 
  Observe that $\hat{\theta}_{track} = \thetaStat$, as desired.
\end{example}

Let $\cEasy$ and $\cStat$ be scalings with positive entries. Let $\FEasy(u,\theta)$ be the difference of the left and right sides of the equations of Definition~\ref{MLequations} with $c$ taken to be $\cEasy$, and similarly for $\FStat(u,\theta)$. We may now state the main result of this section.

\begin{theorem}
Fix a generic data vector $u$ with positive entries. 
Let $\pEasy$ and $\pStat$ be the respective ML estimates for these scaling problems for the data $u$ and consider the homotopy below:
$$
H\left(\theta,t\right) = t\cdot \FEasy(u,\theta) + (1-t)\cdot\FStat(u,\theta).
$$
Let $\gamma$ denote the path of the homotopy whose start point (at $t=1$) corresponds to $\pEasy$.
Then the endpoint of $\gamma$  (at $t=0$) is $\pStat$.
\end{theorem}

\begin{proof}
By Birch's theorem (Theorem~\ref{thm-Birch}), the likelihood equations are given by $A\hat{p} = \frac{1}{u_+}Au$ for a data vector $u$. Let $\hat{p}_{win}$ and $\hat{p}_{stat}$ be the two monomial vectors for $\cEasy$ and $\cStat$ respectively. The homotopy above can be rewritten as $A\cdot(t\hat{p}_{c(t)} +  \frac{1}{u_+}u)$ where $c(t) = {t\cStat + (1-t)\cEasy}$. Since $c(t)$ is positive for any positive real $\cEasy$, $\cStat$ and $t\in[0,1]$, Birch's theorem states that there is exactly one positive real solution to this system at every point along the homotopy path. Thus, as long as no paths intersect, the unique solution for $\cEasy$ will track to the unique positive real solution for $\cStat$. But paths only intersect where the Jacobian of the system drops rank, which we now show cannot occur.

For a value of $\hat p$ at some point on the path, the Jacobian matrix can be written as the product
\[ A \left( \begin{array}{ccc}
\frac{\partial p_1}{\partial \theta_1} & \cdots & \frac{\partial p_1}{\partial \theta_d} \\
\vdots & \ddots & \vdots \\
\frac{\partial p_n}{\partial \theta_1} & \cdots & \frac{\partial p_n}{\partial \theta_d}
\end{array}\right).\]
Because we require the $\theta_i$'s to be nonzero, we can write this as
\[ A \left( \begin{array}{ccc}
\theta_1\frac{\partial p_1}{\partial \theta_1} & \cdots & \theta_d\frac{\partial p_1}{\partial \theta_d} \\
\vdots & \ddots & \vdots \\
\theta_1\frac{\partial p_n}{\partial \theta_1} & \cdots & \theta_d\frac{\partial p_n}{\partial \theta_d}
\end{array}\right)
\operatorname{diag}\left(\frac{1}{\theta_1},\ldots,\frac{1}{\theta_d}\right)
.\]
Note that if we denote $A=(a_{ij})$, then $\theta_i\frac{\partial p_j}{\partial\theta_i}=a_{ij}p_j$, so we may rewrite the middle matrix as 
$\operatorname{diag}(p_1,p_2,\ldots,p_n) A^T$
and the Jacobian thus factors as
\[
A \, \operatorname{diag}(p_1,\ldots,p_n) \,
A^T \, 
\operatorname{diag}\left(\frac{1}{\theta_1},\ldots,\frac{1}{\theta_d}\right)
.\]
This is a product of full rank matrices, hence the Jacobian has full rank.
\end{proof}
Notice that the straight-line homotopy between $F_{easy}$ and $F_{stat}$ is equivalent to a parameter homotopy between $c_{easy}$ and $c_{stat}$; see Example~\ref{ex:veronese}.

\subsection{Timings}
To test the viability of using homotopy methods, we ran timing comparisons with
the iterative proportional scaling algorithm (IPS, Algorithm 2.1.9 in \cite{DSS09}). We
programmed IPS in Python using \texttt{numpy} for fast linear algebra
calculations and a target residual of $\epsilon=10e^{-12}$, while for the
homotopy approach we used Macaulay2 \cite{M2} for setup and PHCpack \cite{Ver99} for the actual path tracking.
We choose the rational normal scroll \eqref{ScrollMatrix} example from Section~\ref{sec:scroll}, setting $n_i=k$ for all $i$ and varying $k\in \{4,\ldots,13\}$ and $d-1 \in \{5,10,15\}$.
For the homotopy start system, we set $c_{ij}=\binom{n_i}{j}$
and tracked from the closed form solution given in Example~\ref{ex:sameN}.
It is clear from Figure~\ref{fig:timings} that as the problem size
grows, the homotopy method has a significant speed advantage over the iterative proportional scaling algorithm.\footnote{Source code can be found at
\href{https://github.com/nbliss/Likelihood-MRC}{https://github.com/nbliss/Likelihood-MRC}}
\begin{figure}
\captionsetup[subfigure]{labelformat=empty}
\centering
\begin{subfigure}{.32\textwidth}
  \centering
  \begin{tikzpicture}[scale=0.50]
\begin{axis}[
xlabel={$n_i$},
ylabel={Time (s)},
ymin=-0.2,ymax=1.5,
scatter/classes={
a5={mark=triangle*,black},
b5={black}
}]
\addplot[scatter,only marks,scatter src=explicit symbolic]  table[meta=label]{data5.dat};
\end{axis}
\end{tikzpicture}
  \caption{$d-1=5$}
\end{subfigure}%
\hspace*{4pt}
\begin{subfigure}{.32\textwidth}
  \centering
  \begin{tikzpicture}[scale=0.50]
\begin{axis}[
xlabel={$n_i$},
ymin=-0.2,ymax=1.5,
ylabel near ticks, 
scatter/classes={
a10={mark=triangle*,black},
b10={black}
}]
\addplot[scatter,only marks,scatter src=explicit symbolic]  table[meta=label]{data10.dat};
\end{axis}
\end{tikzpicture}
  \caption{$d-1=10$}
\end{subfigure}
\hspace*{4pt}
\begin{subfigure}{.32\textwidth}
  \centering
  \begin{tikzpicture}[scale=0.50]
\begin{axis}[
xlabel={$n_i$},
ymin=-0.2,ymax=1.5,
ylabel near ticks, yticklabel pos=right,
scatter/classes={
a15={mark=triangle*,black},
b15={black}
}]
\addplot[scatter,only marks,scatter src=explicit symbolic]  table[meta=label]{data15.dat};
\end{axis}
\end{tikzpicture}
  \caption{$d-1=15$}
\end{subfigure}
\captionsetup{font=scriptsize}
\caption{\label{fig:timings} Running times of iterative proportional scaling (triangles) versus path tracking (circles) on
the rational normal scroll \eqref{ScrollMatrix} for three choices of $d-1$ with $n_i$'s constant. Average of 7 trials.}
\end{figure}
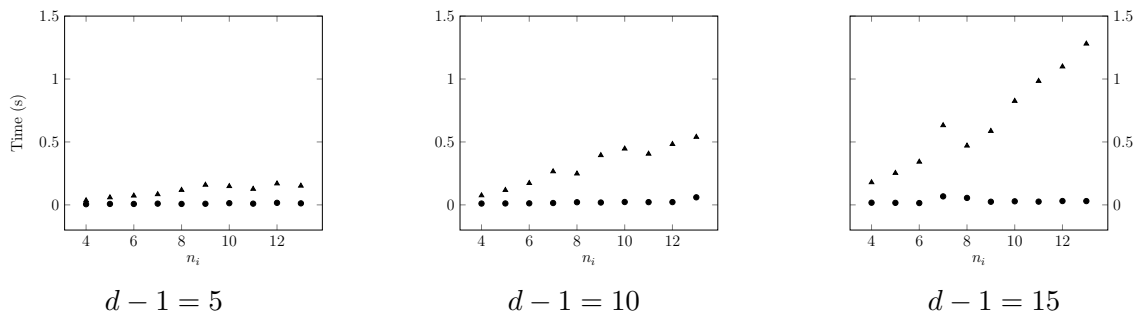

\bibliographystyle{amsalpha}
\bibliography{bibfile}

\end{document}